\documentclass[11pt]{amsart}
\usepackage[a4paper,top=2.5cm,bottom=2.5cm,left=2.5cm,right=2.5cm]{geometry}

\usepackage{amsmath}
\usepackage{amsthm}
\usepackage{amssymb}

\usepackage{breqn}

\usepackage{graphicx}
\usepackage[colorlinks=true, allcolors=blue]{hyperref}

\newtheorem{theorem}{Theorem}[section]
\newtheorem{proposition}[theorem]{Proposition}
\newtheorem{lemma}[theorem]{Lemma}

\newtheorem{conjecture}[theorem]{Conjecture}

\newtheorem*{remark}{Remark}

\title{The Alexander polynomial of twisted torus knots in the complementary range}

\author{Adnan}
\address{Department of Mathematics, Kangwon National University, Chuncheon-si, Gangwon-do, 24341, Republic of Korea}
\email{adnanshahab35@kangwon.ac.kr}

\author{Kyungbae Park}
\address{Department of Mathematics, Kangwon National University, Chuncheon-si, Gangwon-do, 24341, Republic of Korea}
\email{kyungbaepark@kangwon.ac.kr}


\begin{document}

\begin{abstract}
A twisted torus knot $T(p,q;r,s)$ is obtained from the flat-torus diagram
of $T(p,q)$ by inserting $s$ full twists along a disk meeting $r$ adjacent
strands.  Our previous work determined its Alexander polynomial for
$1<r\leq\max\{|p|,|q|\}$.  In this paper, we compute it in the
complementary range
$\max\{|p|,|q|\}<r\leq|p|+|q|$ by placing the knot on a standard genus-two
Heegaard surface and applying Fox calculus.  We then combine the two
calculations into a unified Alexander-polynomial formula valid for
$1<r\leq|p|+|q|$.  In the complementary range, we prove monicity and give an explicit
formula for the breadth.  This determines the Seifert genus for
nonnegative twists and gives a genus lower bound for negative twists.  We
also exhibit infinite two-parameter families of non-$L$-space knots in the
complementary range.
\end{abstract}


\maketitle

\section{Introduction}
Let $p$ and $q$ be relatively prime integers and let $s\in\mathbb Z$.  A
twisted torus knot $T(p,q;r,s)$ is obtained from the flat-torus diagram of
the torus knot $T(p,q)$ by inserting $s$ full twists along a disk meeting
$r$ adjacent strands, where the original genus-two tangle construction
allows $1<r\leq |p|+|q|$~\cite{Dean-1996,Dean-2003}.  See
Figure~\ref{fig:intro-genus-two} for the example $T(7,5;9,1)$ on a
standard genus-two Heegaard surface.  Introduced in the
study of Seifert fibered Dehn surgery, twisted torus knots have since
found many applications in low-dimensional topology.  Representative
examples include hyperbolic knot theory
\cite{CDW-1999,CKP-2004}, Heegaard splittings
\cite{Moriah-Sedgwick-2009}, and $L$-space surgery
\cite{Vafaee-2015,CGHV-2016,Tran-2020}.

We use the following elementary symmetries to normalize the parameters.
We write $K\sim K'$ if $K$ and $K'$ are ambient isotopic as unoriented knots
in $S^3$.  Exchanging the two solid tori bounded by the flat torus, reversing
the orientation of the torus knot, and reflecting $S^3$, respectively, give
\begin{equation*}
T(p,q;r,s)\sim T(q,p;r,s),\qquad
T(-p,-q;r,s)\sim T(p,q;r,s),\qquad
\overline{T(p,q;r,s)}\sim T(p,-q;r,-s).
\end{equation*}
These symmetries follow directly from the geometric definition and are used
with the same conventions in~\cite{Lee-2015,Lee-2021}; see
also~\cite{CFKNP-2011}.  Since the Alexander polynomial is unchanged by
reversing orientation or taking the mirror image, we assume
throughout that $p>q>0$, while retaining arbitrary $s\in\mathbb Z$.  We call
$T(p,q;r,s)$ \emph{positively twisted} when $s>0$.

When $r\leq p$, the twisting operation is visible in the standard closed
$p$-braid diagram.  Our previous formula treats $1<r<p$
\cite{Park-Adnan-2025-01}; adjoining the elementary torus-knot case $r=p$
gives the range $1<r\leq p$.  This restriction belongs to the braid model
rather than to the definition.
Dean's construction includes the full range up to $p+q$, and
Champanerkar, Futer, Kofman, Neumann, and Purcell likewise assume
$p,q>0$ and allow $0<r\leq p+q$ in their generalized
twisted-torus-link construction
\cite{CFKNP-2011}.  Moreover, the range $p<r\leq p+q$ occurs explicitly
in questions about knot type and braid index: Lee constructed closed braids
for $r>p$~\cite{Lee-2015} and later studied precisely
$1<q<p<r\leq p+q$~\cite{Lee-2021}.  Thus the general range is both natural
and geometrically relevant, but its Alexander polynomial is not directly
accessible from the standard $p$-strand braid presentation.

\begin{figure}[t]
    \centering
    \includegraphics[width=0.8\textwidth]{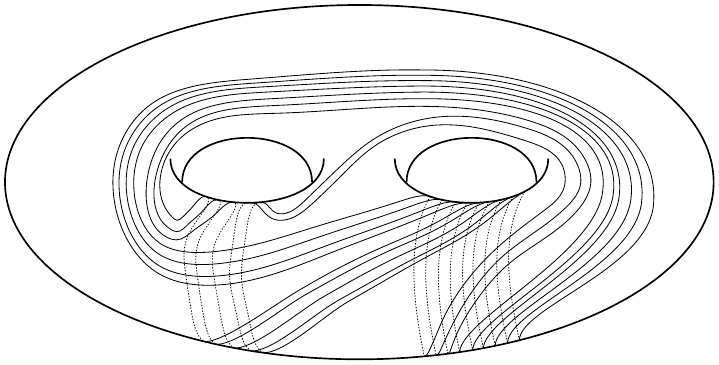}
    \caption{The twisted torus knot $T(7,5;9,1)$ on a standard genus-two
    Heegaard surface.}
    \label{fig:intro-genus-two}
\end{figure}

For $1<r\leq p$, the classical braid description is the closure of
\[
(\sigma_1\cdots\sigma_{p-1})^q
(\sigma_1\cdots\sigma_{r-1})^{rs}.
\]
Morton computed the Alexander polynomial for $r=2$ using the Burau
representation and the skein relation~\cite{Morton-2006}.  For $r=2$, the other diagrams in the skein relation remain twisted
torus knots or links, allowing a recurrence within that class.  For
$r>2$, these diagrams need not be twisted torus knots or links, so the
same recurrence does not directly apply.  Related work
includes Vafaee's computation of knot Floer homology for certain families
with $p=qk\pm1$~\cite{Vafaee-2015} and the computation of the Jones
polynomial for $r=2$~\cite{BD-2023}.  The genus-two approach used here
replaces the braid restriction by a direct calculation of the inclusion
words and their Fox derivatives.

The purpose of this paper is to complete and unify the Alexander-polynomial
calculation.  Write $f\doteq g$ for equality up to multiplication by a
unit $\pm t^k$ in $\mathbb Z[t^{\pm1}]$.  For a positive integer $n$, let
$[a]_n$ denote the least nonnegative residue of an integer $a$ modulo $n$.
For the main formula, assume $p>q>1$, $\gcd(p,q)=1$, and
$1<r\leq p+q$, and let $p^{-1}$ denote the inverse of $p$ modulo $q$.  Set
\begin{equation}\label{eq:unified-indices}
j_*=[-rp^{-1}]_q,
\qquad
i_*=p+1-\frac{pj_*+r}{q}.
\end{equation}
Here $0\leq j_*<q$ and $\ell_*=(pj_*+r)/q$ is an integer satisfying
$1\leq\ell_*\leq p+1$, since $0<pj_*+r\leq q(p+1)$.  Thus
$0\leq i_*\leq p$; equality $i_*=0$ requires equality in the upper bound,
which occurs exactly when $j_*=q-1$ and $r=p+q$.
For $0\leq j<q$ and $0\leq i<p$, define the
residue-block lengths
\[
\mu_j=[-jp]_q,\qquad
L_j=\left\lfloor\frac{r-1-\mu_j}{q}\right\rfloor+1,
\]
\[
\tau_i=[-iq]_p,\qquad
M_i=\left\lfloor\frac{r-1-\tau_i}{p}\right\rfloor+1.
\]
Write
\[
\overline L_m=\sum_{j=0}^mL_j,\qquad
\overline M_m=\sum_{i=0}^mM_i,
\qquad
\overline L_{-1}=\overline M_{-1}=0,
\]
and define the adjusted block lengths by
\[
\widehat M_0=M_0-1,\qquad
\widehat M_i=M_i\quad(1\leq i<p).
\]
The adjustment affects only the endpoint $r=p+q$: for $r<p+q$, we have
$i_*\geq1$, so all block lengths appearing in $E_a$ and $Z_a$ below
are unchanged.  Set
\[
E_a=q(p-i_*+1)+rs\sum_{i=i_*}^{p-1}\widehat M_i,
\qquad
E_c=r^2s+pq.
\]

\begin{theorem}\label{thm:alexander-polynomial}
Let $p>q>1$ be relatively prime, let $1<r\leq p+q$, and let
$s\in\mathbb Z$.  Define the four residue-block Laurent polynomials by
\begin{equation}\label{eq:block-polynomials}
\begin{aligned}
X_a(t)&=\sum_{m=0}^{j_*-1}t^{rs\overline L_m+pm},
&
X_c(t)&=\sum_{m=0}^{q-1}t^{rs\overline L_m+pm},\\
Z_a(t)&=t^{E_a}\sum_{h=0}^{p-i_*}
t^{-qh-rs\sum_{i=p-h}^{p-1}\widehat M_i},
&
Z_c(t)&=t^{E_c}\sum_{m=0}^{p-1}t^{-rs\overline M_m-qm}.
\end{aligned}
\end{equation}
Here and throughout, an empty sum is understood to be zero.
Then
\[
\Delta_{T(p,q;r,s)}(t)
\doteq
\frac{1-t}{1-t^r}
\bigl(Z_a(t)X_c(t)-X_a(t)Z_c(t)\bigr).
\]
\end{theorem}

\begin{remark}
For $1<r<p$, away from the separately evaluated case $r=q$, this is
equivalent to the formula of~\cite{Park-Adnan-2025-01}; the range $p<r<p+q$ is the new calculation in
this paper.  At $r=q$ and $r=p$ it specializes respectively to the
Alexander polynomials of the torus knots $T(q,p+qs)$ and $T(p,q+ps)$.
The formula is valid for positive, zero, and negative $s$.
For $r<p+q$, one has $i_*\geq1$, so the adjustment of $M_0$ does
not enter $E_a$ or $Z_a$ and the strict-range formula is unchanged.
At $r=p+q$, the adjusted sequence accounts for the terminal return word;
Proposition~\ref{prop:endpoint-formula} verifies this case of the same formula.
\end{remark}

For $q=1$, the strict complementary range is empty, while the endpoint
$r=p+1$ reduces to a torus-knot calculation.  The cases $q=1$ and $r=1$
are explained in Section~\ref{sec:elementary-cases}.

The formula has several applications in the complementary range.  The
\emph{breadth} of a nonzero Laurent polynomial is the difference between
its highest and lowest exponents, equivalently its degree after
normalizing the lowest exponent to zero.  We determine this breadth as
follows.

\begin{theorem}\label{thm:alexander-degree}
Let $K=T(p,q;r,s)$ satisfy $p<r\leq p+q$, $p>q>0$, and
$\gcd(p,q)=1$.  Then
\[
\operatorname{br}\Delta_K(t)=
\begin{cases}
(p-1)(q-1)+sr(r-1),&s\geq0,\\[2mm]
(-s)r(r-1)-(p-1)(q-1)-2(r-1),&s<0.
\end{cases}
\]
Moreover, if $s\geq0$, then $K$ is fibered and
\[
g(K)=\frac{(p-1)(q-1)+sr(r-1)}{2}.
\]
\end{theorem}

The classical Alexander genus bound gives
\[
g(K)\geq\frac12\operatorname{br}\Delta_K(t),
\]
with equality for fibered knots.  For $s>0$ and $p<r<p+q$, $K$ can be
written as the closure of a positive $r$-strand braid by the isotopy in
\cite[Figures~3(b) and~4]{Lee-2015}.  At $r=p+q$, the standard
$(p+q)$-strand Lorenz braid for $T(p,q)$, followed by $s$ positive full
twists on all strands, gives the same conclusion.  Thus $K$ is fibered by
Stallings' theorem~\cite{Stallings-1978}.  For $s=0$, $K$ is the torus
knot $T(p,q)$.
Thus the breadth formula determines the genus for $s\geq0$.
For negative twists, we first consider monicity of the Alexander
polynomial, a necessary condition for fiberedness.

\begin{theorem}\label{thm:monicity}
Let $p<r\leq p+q$, $p>q>0$, $\gcd(p,q)=1$, and $s\in\mathbb Z$.
Then the Alexander polynomial of $T(p,q;r,s)$ is monic.
\end{theorem}

This conclusion contrasts with the range $r<p$, where in our recent
work~\cite{Adnan-Park-2026} we used nonmonic Alexander polynomials to
construct several infinite families of non-fibered twisted torus knots.
For $s<0$, monicity alone does not imply fiberedness, so we examined
knot Floer homology for further evidence.  Using Spherogram~2.3 and the
Python wrapper~1.2.2 for Szab\'o's HFK Calculator
\cite{Spherogram,HFK-Calculator}, we computed knot Floer homology over
$\mathbb F_2$ for all $26$ parameter choices satisfying
\[
3\leq p\leq6,\qquad 2\leq q<p,\qquad \gcd(p,q)=1,
\qquad p<r<p+q,\qquad s\in\{-2,-1\}.
\]
Every knot had rank one in the top Alexander grading and hence was
fibered by Ni's theorem~\cite{Ni-2007}.  In each case, the graded Euler
characteristic agreed with Theorem~\ref{thm:alexander-polynomial}, and
the Seifert genus detected by knot Floer homology equaled one half of the
breadth in Theorem~\ref{thm:alexander-degree}.  Thus the genus bound is
sharp in these negative-twist examples.  These computations motivate the
following conjecture.

\begin{conjecture}\label{conj:fiberedness}
Every twisted torus knot $T(p,q;r,s)$ with
$p+q>r>p>q>0$, $\gcd(p,q)=1$, and $s<0$ is fibered.
\end{conjecture}

By Theorem~\ref{thm:monicity}, every knot in the range of
Conjecture~\ref{conj:fiberedness} has a monic Alexander polynomial.  Thus
Conjecture~\ref{conj:fiberedness} is the restriction to the complementary
negative-twist range of the ``monic implies fibered'' direction of our
earlier conjecture~\cite[Conjecture~1.6]{Adnan-Park-2026}.  The knot Floer
homology computations above provide evidence for this restricted case,
while we continue to expect the earlier conjecture to hold in its original
generality.  Moreover, in the complementary range $p<r<p+q$ with $s<0$,
Conjecture~\ref{conj:fiberedness}, together with
Theorem~\ref{thm:alexander-degree}, predicts that the Alexander genus bound
is sharp.  We pursue a braid-theoretic approach to this conjecture in a
separate work in progress.

An $L$-space knot is a knot in $S^3$ admitting a nontrivial Dehn
surgery to an $L$-space.  Introduced by Ozsv\'ath and
Szab\'o~\cite{Ozsvath-Szabo-2005}, an $L$-space is a rational homology
sphere $Y$ whose Heegaard Floer homology has the smallest possible rank,
namely $\dim_{\mathbb F_2}\widehat{HF}(Y;\mathbb F_2)=|H_1(Y;\mathbb Z)|$.
The knot Floer homology constraints of
\cite[Theorem~1.2]{Ozsvath-Szabo-2005}, together with Ni's
theorem~\cite{Ni-2007}, imply that every $L$-space knot is fibered,
but the converse does not hold.  More specifically, Ozsv\'ath and
Szab\'o~\cite[Corollary~1.3]{Ozsvath-Szabo-2005} proved that the nonzero
Alexander-polynomial coefficients of an $L$-space knot are $\pm1$ and
alternate in sign.  Thus a coefficient of absolute value greater than
one obstructs the $L$-space property even for a fibered knot.
In our previous work~\cite{Park-Adnan-2025-01}, we used this obstruction
to construct infinite families of non-$L$-space twisted torus knots
with $r<p$.  Applying the same criterion in the complementary range
$p<r<p+q$ gives the following two-parameter families.

\begin{theorem}\label{thm:non-l-space-family}
For positive integers $n$ and $u$, the following twisted torus knots are
not $L$-space knots:
\begin{enumerate}
\item $T(2n+1,2n;3n,-u)$ for $n\geq3$;
\item $T(2n+1,2n;3n,u)$ for $n\geq4$ and $u\geq2$, and
      $T(2n+1,2n;3n,1)$ for $n\geq5$.
\end{enumerate}
All the knots listed above are pairwise distinct.
\end{theorem}

After normalizing each Alexander polynomial to have lowest exponent zero
and positive constant term, the proof explicitly locates the first
coefficient of absolute value greater than one.  In the three respective
cases above, this coefficient occurs at exponent $4(3u-2)n$,
$(15u+9)n+3$, and $29n+4$, and equals $2$, $-2$, and $-2$,
respectively.  These
data, together with the Alexander breadth, also distinguish all the knots
in the theorem.

Finally, we return to the question from our previous
work~\cite{Park-Adnan-2025-01} of whether the Alexander polynomial
determines a positively twisted torus knot up to mirror image.  Extending
the previous computations to the complementary range, we examined $12{,}435$
parameter choices satisfying
\[
3\leq p\leq20,\quad 0<q<p,\quad \gcd(p,q)=1,\quad
0<r<p+q,\quad 1\leq s\leq5.
\]
In the $258$ equal-Alexander groups with at least two numerically
hyperbolic members, the volumes agree within $10^{-8}$ and the decorated
link isometry signatures coincide.  Conversely, no equal-volume pair
with different Alexander polynomials was found among the $7{,}504$
numerically hyperbolic parameter choices.  These finite numerical checks
with SnapPy~3.2~\cite{SnapPy} motivate the following extension of our
earlier question.  The mirror of a positively twisted torus knot need not
itself be positively twisted.  Nevertheless, since neither the Alexander
polynomial nor the numerical invariants used here distinguish mirror
images, we formulate the conjecture up to mirroring.

\begin{conjecture}\label{conj:alexander-completeness}
Let $K=T(p,q;r,s)$ and $K'=T(p',q';r',s')$ be twisted torus knots, where
$p>q>0$, $\gcd(p,q)=1$, $0<r<p+q$, and $s>0$, and where the primed
parameters satisfy the analogous conditions.  If
$\Delta_K(t)\doteq\Delta_{K'}(t)$, then $K$ is isotopic in $S^3$ to either
$K'$ or its mirror image $\overline{K'}$.
\end{conjecture}

Section~2 derives the knot group in the complementary range, and Section~3
proves the unified formula, compares it with the previous formula, and treats
the endpoint $r=p+q$.  Section~4 proves monicity and determines the Alexander
breadth and Seifert genus in the complementary range.  Section~5 constructs
the two-parameter families of non-$L$-space knots.
Code and data for the computations supporting the conjectures above are
available in~\cite{Park-Adnan-Code}.

\subsection*{Use of AI tools}
AI tools assisted with checking calculations, developing and testing code,
and improving the English and presentation.  The authors take full
responsibility for the mathematical content, computational results, and
final manuscript.

\subsection*{Acknowledgments}
The authors were supported by National Research Foundation of Korea grants
funded by the Korean government (RS-2025-24523511 and RS-2025-25415913).

\section{The knot group}

In this section, we derive a presentation of the knot group of twisted
torus knots in the complementary range $p<r<p+q$.  The discussion follows
the same general strategy as our previous work~\cite{Park-Adnan-2025-01}:
we place the knot on a standard genus-two Heegaard surface, encode the
inclusion words of three distinguished loops, and apply the
Seifert--Van Kampen theorem.  In the complementary range, however, the
inclusion words require a different combinatorial description, which we
formulate in terms of residue-block decompositions.  The resulting
presentation is arranged for the Fox-calculus computation in Section~3.

\begin{figure}[t]
    \centering
    \includegraphics[width=0.7\textwidth]{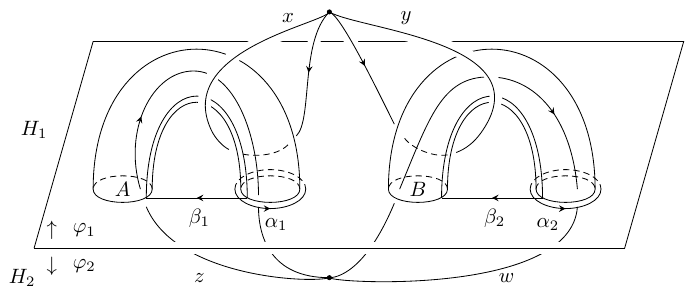}
    \caption{The standard genus-two Heegaard decomposition
    $S^3=H_1\cup H_2$, with generators $x,y$ of $\pi_1(H_1)$ and $z,w$ of
    $\pi_1(H_2)$.}
    \label{fig:handlebody-generators}
\end{figure}

\begin{figure}[t]
    \centering   
    \includegraphics[width=0.95\textwidth]{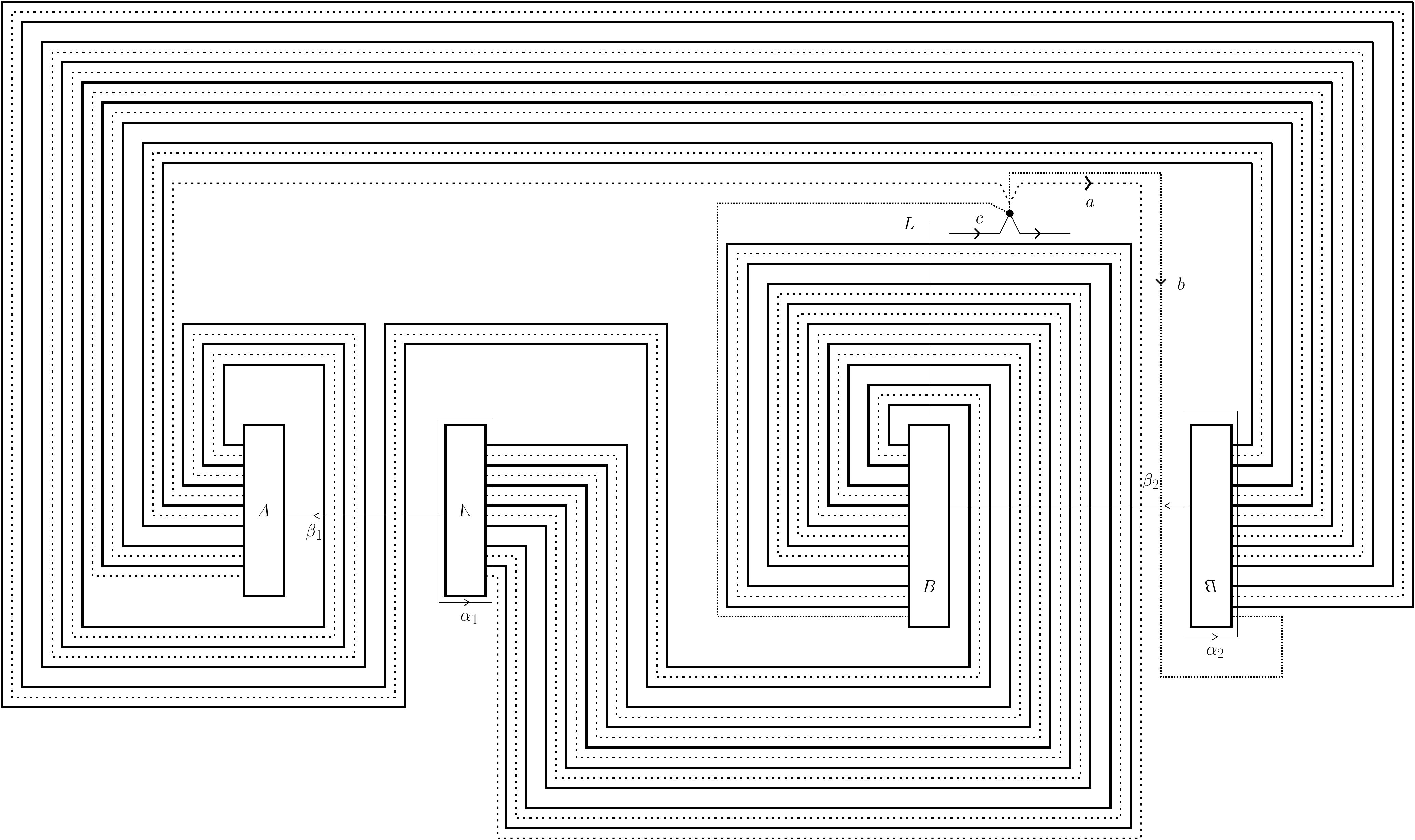}
    \caption{The twisted torus knot $T(7,5;9,1)$ on the standard genus-two
    Heegaard surface, together with the loops $a$, $b$, and $c$ used in the
    Seifert--Van Kampen presentation.}
    \label{fig:heegaard-loops}
\end{figure}

\begin{figure}[t]
    \centering
    \includegraphics[width=0.18\textwidth]{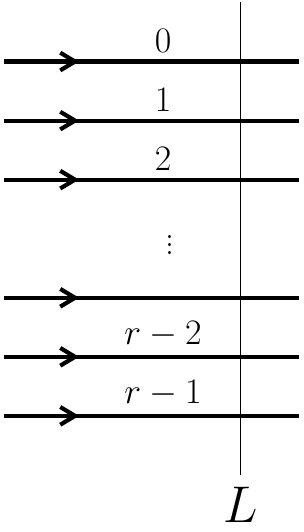}
    \caption{The $r$ sectors of the twisting disk $\mathcal D$, indexed by
    $0,1,\ldots,r-1$.}
    \label{fig:sectors}
\end{figure}

\begin{theorem}\label{thm:knot-group}
Let $T(p,q;r,s)$ be a twisted torus knot with $\gcd(p,q)=1$,
$p+q>r>p>q>0$, and $s\in\mathbb Z$.  With the residue indices
$j_*,i_*$ and block lengths $L_j,M_i$ defined above, its knot group has
the following presentation:
\[
\pi_1(S^3\setminus\nu(K))
\cong\langle x,y,z,w\mid r_1,r_2,r_3\rangle,
\qquad K=T(p,q;r,s),
\]
where
\[
\left\{
\begin{aligned}
r_1&=\left(\prod_{j=0}^{j_*-1}y^{L_js}x\right)y^{(L_{j_*}-1)s+1}
\left(\left(\prod_{k=i_*}^{p-1}z^{-1}w^{-M_k}\right)z^{-1}\right)^{-1},\\[3pt]
r_2&=y^sw,\\[3pt]
r_3&=y^{L_0s}x\,y^{L_1s}x\,\cdots\,y^{L_{q-1}s}x
\cdot w^{M_0}zw^{M_1}z\cdots w^{M_{p-2}}zw^{M_{p-1}}z.
\end{aligned}
\right.
\]
\end{theorem}

\begin{proof}
Place $K=T(p,q;r,s)$ on the standard genus-two Heegaard surface
$\Sigma$ in $S^3=H_1\cup_\Sigma H_2$.  Removing a sufficiently small open
tubular neighborhood $\nu(K)$ gives
\[
S^3\setminus\nu(K)=H_1'\cup_{\Sigma'}H_2',\qquad
H_i'=H_i\setminus\nu(K),\quad \Sigma'=\Sigma\setminus\nu(K).
\]
The knot is nonseparating on $\Sigma$, so $\Sigma'$ is a genus-one
surface with two boundary components and has free fundamental group of
rank three.  Each $H_i'$ is homeomorphic to a genus-two handlebody and
has free fundamental group of rank two.  We use generators $x,y$ and
$z,w$ on the two sides, as in Figure~\ref{fig:handlebody-generators},
and the three loops $a,b,c$ on $\Sigma'$ shown in
Figure~\ref{fig:heegaard-loops}.  The loop $c$ follows a parallel copy of the knot; $a$ and $b$ close
the indicated subarcs along the return arcs shown in that figure.  Write
$\varphi_1$ and $\varphi_2$ for the inclusion homomorphisms into the
two handlebodies, as shown in Figure~\ref{fig:handlebody-generators}.  The Seifert--Van Kampen theorem gives
\[
\pi_1(S^3\setminus\nu(K))
\cong
\langle x,y,z,w\mid \varphi_1(a)\varphi_2(a)^{-1},\ \varphi_1(b)\varphi_2(b)^{-1},\ \varphi_1(c)\varphi_2(c)^{-1}\rangle.
\]
We now compute these six inclusion words.  Let $\mathcal D$ denote the twisting disk
and index its $r$ sectors by $0,1,\ldots,r-1$, as shown in
Figure~\ref{fig:sectors}.

Starting with $u_0=0$, define the sequence $\{u_n\}_{n=0}^{r-1}$
by the following recurrence:
\begin{equation*}
    u_{n+1}=
    \begin{cases}
        u_n+q&\text{if }0\leq u_n< r-q\\
        [u_n-p]_q&\text{if } r-q \leq u_n<r
    \end{cases}
\end{equation*}
Use the residues $\mu_j$ and block lengths $L_j$ defined above.
Retain the indices from Equation~\eqref{eq:unified-indices}.  Under the
present assumption $p<r<p+q$, they satisfy
\[
0\leq j_*\leq q-2,
\qquad
2\leq i_*\leq p-1.
\]
Indeed, $i_*=p$ would give $pj_*+r=q$, contrary to $r>q$.
Thus $0\leq i_*<p$, and Equation~\eqref{eq:unified-indices} gives
\[
-i_*q=-pq-q+pj_*+r\equiv r-q\pmod p.
\]
Since $p-q<r-q<p$, it follows that
\begin{equation}\label{eq:truncated-residue}
\tau_{i_*}=r-q.
\end{equation}
This excludes $i_*=0$ and $i_*=1$, because $\tau_0=0$ and
$\tau_1=p-q$.
Finally, $j_*=q-1$ would imply $r\equiv p\pmod q$, impossible because
$0<r-p<q$.
Set $c_j=\overline L_{j-1}$ for $0\leq j\leq q$.  Thus $c_0=0$,
and the residue blocks partition $\{0,\ldots,r-1\}$, giving $c_q=r$.

For each
\[n\in\{0,1,\ldots,r-1\},\]
let $j\in \{0,1,\ldots,q-1\}$ be the unique index such that
\[c_j \le n < c_{j+1}.\]
Equivalently, the sequence $u_n$ is given by
\[u_n= \mu_j + q(n-c_j).\]

\begin{lemma}\label{lem:u-blocks}
The recursive definition of $u_n$ agrees with the block description above.
Moreover, $u_0,u_1,\ldots,u_{r-1}$ are precisely the integers
$0,1,\ldots,r-1$, each occurring once.
\end{lemma}

\begin{proof}
For a fixed $j$, the integers in the $j$-th block are
\[
\mu_j,\ \mu_j+q,\ \ldots,\ \mu_j+(L_j-1)q.
\]
Every element except the last one is less than $r-q$, so the recursion adds
$q$.  The last element lies in $[r-q,r-1]$, and the next value is
\[
\bigl[\mu_j+(L_j-1)q-p\bigr]_q=[\mu_j-p]_q=\mu_{j+1},
\]
where the index $j+1$ is read modulo $q$.  Thus the recursion concatenates
the stated blocks and returns to $\mu_0=0$ after the last block.  Since the
$\mu_j$ form all residue classes modulo $q$, these blocks partition
$\{0,1,\ldots,r-1\}$.
\end{proof}

For the segment of $c$ beginning with sector label $u_n$, the contribution
to the word in $\pi_1(H_1)$ is given by the following rule:
\begin{equation*}
    \iota_n=
    \begin{cases}
        y^s&\text{if } u_n< r-q\\
        y^sx&\text{if } r-q \leq u_n<r
    \end{cases}
\end{equation*}
Therefore, Lemma~\ref{lem:u-blocks} gives
\[
\varphi_1(c)=\prod_{n=0}^{r-1}\iota_n
=\prod_{j=0}^{q-1}y^{L_js}x.
\]
Explicitly,
\[\varphi_1(c)=y^{L_0s}x\,y^{L_1s}x\,\cdots\,y^{L_{q-1}s}x.\]
We next read the same sectors from the $H_2$ side using the residues
$\tau_i$ and block lengths $M_i$ defined above.

\begin{lemma}\label{lem:M-blocks}
When the sectors of $\mathcal D$ are read from the $H_2$ side in the orientation
shown in Figure~\ref{fig:heegaard-loops}, they form $p$ consecutive blocks,
in the order
\[
\mathcal B_{p-1},\mathcal B_{p-2},\ldots,\mathcal B_0,
\]
where $\mathcal B_i$ contains $M_i$ sectors.  The first segment in
$\mathcal B_i$ contributes $z^{-1}w^{-1}$ to $\varphi_2(c)$, and each of its
remaining $M_i-1$ segments contributes $w^{-1}$.  Consequently,
\[
\varphi_2(c)=\prod_{i=0}^{p-1}z^{-1}w^{-M_{p-1-i}}.
\]
\end{lemma}

\begin{proof}
For each $i$, the representatives of the residue class $\tau_i$ that lie in
$\{0,1,\ldots,r-1\}$ are
\[
\tau_i,\ \tau_i+p,\ \ldots,\ \tau_i+(M_i-1)p;
\]
thus their number is exactly $M_i$.  Since the $\tau_i$ run through all
residue classes modulo $p$, these sets partition the $r$ sectors.
Crossing from one block to the next adds $q$ to the sector label modulo
$p$.  Hence the relation
\[
\tau_{i-1}=[\tau_i+q]_p
\]
shows that the blocks occur in decreasing order of their indices when the
boundary of the twisting disk is followed from the $H_2$ side.  The chosen
base point puts $\mathcal B_{p-1}$ first.  As seen in
Figure~\ref{fig:heegaard-loops}, entering a block crosses the disk dual to
$z$, while every sector in that block crosses the disk dual to $w$.  This
gives the stated contributions and the displayed word.
\end{proof}

\begin{lemma}\label{lem:M-palindrome}
Let $\delta=r-p$.  Then $M_i\in\{1,2\}$, with
\[
M_i=2\quad\Longleftrightarrow\quad \tau_i<\delta.
\]
Moreover, the finite sequence
\[
M_{i_*},M_{i_*+1},\ldots,M_{p-1}
\]
is a palindrome.
\end{lemma}

\begin{proof}
The first assertion follows immediately from $r=p+\delta<2p$.
Recall that $2\leq i_*\leq p-1$ in the strict range $p<r<p+q$.
For $0\leq h\leq p-i_*-1$, both $i_*+h$ and $p-1-h$ lie in
$\{i_*,\ldots,p-1\}$.  They may coincide at the middle term, where
the required equality is automatic.  Set
\[
\omega_h=\tau_{p-1-h}=[(h+1)q]_p.
\]
Since $\tau_{i_*}=p+\delta-q$, we have
\[
\tau_{i_*+h}=[\delta-\omega_h]_p.
\]
Since the residues $\tau_i$ are distinct and $\tau_0=0$, the index bounds
imply $\omega_h\ne0$.  Also, $\tau_{i_*-1}=[\tau_{i_*}+q]_p=\delta$, so
$\omega_h=\delta$ would require $p-1-h=i_*-1$, or $h=p-i_*$, outside
the stated range.  Therefore,
\[
\tau_{i_*+h}<\delta
\quad\Longleftrightarrow\quad
\omega_h<\delta
\quad\Longleftrightarrow\quad
\tau_{p-1-h}<\delta.
\]
The first assertion now gives $M_{i_*+h}=M_{p-1-h}$.
\end{proof}

The word in Lemma~\ref{lem:M-blocks} can also be written as
\[\varphi_2(c) = z^{-1}w^{-M_{p-1}}z^{-1}w^{-M_{p-2}}\cdots z^{-1}w^{-M_0}.\]

Loop $a$ follows the same path as loop $c$ until it reaches the value
$u_n=r-q$, and then returns to the base point.  The index $j_*$ from
Equation~\eqref{eq:unified-indices} is characterized by
\[
\mu_{j_*}\equiv r-q \pmod q,
\qquad\text{or equivalently}\qquad
j_*\equiv-rp^{-1}\pmod q.
\]
Since $r-q$ is the largest integer less than $r$ in its residue class modulo
$q$, we have
\[
r-q=\mu_{j_*}+(L_{j_*}-1)q.
\]
Thus, in the $j_*$-th block, the loop $a$ traverses $L_{j_*}-1$ ordinary
segments, each contributing $y^s$.  The final return segment crosses the meridian disk dual to $y$ once
and contributes $y$.  Put
\[
\lambda_*=(L_{j_*}-1)s+1.
\]
The word associated to the truncated loop $a$ is
\[\varphi_1(a)=\left(\prod_{j=0}^{j_*-1}y^{L_js}x\right)y^{\lambda_*}.
\]
Equivalently,
\[
\varphi_1(a)=y^{L_0s}x\,y^{L_1s}x\,\cdots\,
y^{L_{j_*-1}s}x\,y^{(L_{j_*}-1)s+1}.
\]
The initial part of $\varphi_2(c)$ followed by the return segment gives
\[
\varphi_2(a)=z^{-1}w^{-M_{p-1}}\cdots z^{-1}w^{-M_{i_*}}z^{-1}.
\]
By Lemma~\ref{lem:M-palindrome}, this can equivalently be written as
\[
\varphi_2(a)=\left(\prod_{k=i_*}^{p-1}z^{-1}w^{-M_k}\right)z^{-1},
\]
where $i_*$ is the index fixed in Equation~\eqref{eq:unified-indices}.

Consider the homomorphism from the free group on $x,y,z,w$ to the infinite
cyclic group generated by $t$ given by
\[
\alpha(x)=t^p,\qquad
\alpha(y)=t^r,\qquad
\alpha(z)=t^{-q},\qquad
\alpha(w)=t^{-rs}.
\]

\begin{lemma}\label{lem:truncation-identities}
For the residue data defined above and every $1<r<p+q$, the
truncation indices and block lengths satisfy
\[
\sum_{j=0}^{j_*}L_j=1+\sum_{k=i_*}^{p-1}M_k
\]
and
\[
pj_*+r=q(p-i_*+1).
\]
In the complementary range, the two words associated to $a$ therefore
have the same image under abelianization.
\end{lemma}

\begin{proof}
The second identity is Equation~\eqref{eq:unified-indices} rearranged.
For the first, the following lattice-point argument applies throughout
$1<r<p+q$, including the earlier range $r<p$.  Put
$\ell_*=(pj_*+r)/q=p-i_*+1$ and consider integer pairs $(j,\ell)$ satisfying
\[
0\leq q\ell-pj\leq r-1.
\]
For each $0\leq j\leq j_*$, there are $L_j$ such pairs, and their
second coordinates lie in $0\leq\ell\leq\ell_*-1$.  Among the pairs with
$0\leq j\leq j_*$, the only one with $\ell=0$ is $(0,0)$.  Conversely, for $1\leq\ell\leq\ell_*-1$, the inequalities
force $0\leq j\leq j_*$: if $j\leq-1$, then
$q\ell-pj\geq p+q>r-1$; if $j\geq j_*+1$, then
$q\ell-pj\leq r-p-q<0$.
For a fixed such $\ell$, the number of pairs is $M_{p-\ell}$, since
$q\ell-pj$ runs through the representatives of $\tau_{p-\ell}$ in
$\{0,\ldots,r-1\}$.  Counting by the two coordinates and removing
$(0,0)$ gives the first identity.  Finally, if
$A=\sum_{j=0}^{j_*}L_j$ and $B=\sum_{k=i_*}^{p-1}M_k$, then
\[
\begin{aligned}
\alpha(\varphi_1(a))
&=t^{pj_*+r\bigl(s(A-1)+1\bigr)},\\
\alpha(\varphi_2(a))
&=t^{q(p-i_*+1)+rsB}.
\end{aligned}
\]
The two identities above show that these powers are equal.
\end{proof}

Reading the loop $b$ against the meridian disks dual to the generators
in Figure~\ref{fig:heegaard-loops} gives
\[
\varphi_1(b)=y^s,\qquad \varphi_2(b)=w^{-1}.
\]
Substituting the six words above into the Seifert--Van Kampen presentation
gives the three relators in the statement.
\end{proof}

\section{Fox derivatives and the Alexander polynomial}

In this section, we prove Theorem~\ref{thm:alexander-polynomial}.  Starting
from the presentation obtained in Section~2, we compute the relevant Fox
derivatives and express their abelianizations as four block polynomials.
We then evaluate a maximal minor of the Alexander matrix, compare the
result with the formula from~\cite{Park-Adnan-2025-01}, and treat the three
boundary cases $r=q$, $r=p$, and $r=p+q$.

\subsection{The Alexander matrix}

The group calculation in Section~2 assumes $p<r<p+q$.  For the unified
statement, however, the residue formulas defining $\mu_j,L_j,\tau_i,M_i$ are
used for every $1<r\leq p+q$, with $j_*$ and $i_*$ chosen by
Equation~\eqref{eq:unified-indices}.  Some block lengths may be zero when
$r<p$; all empty sums and products are interpreted as zero and one,
respectively.  In the complementary range these conventions agree with the
geometric definitions in Section~2.  We first compute the derivatives for
$p<r<p+q$; the endpoint is verified in
Proposition~\ref{prop:endpoint-formula}.

We use the homomorphism $\alpha$ introduced in the proof of
Theorem~\ref{thm:knot-group}.  Geometrically, its exponents are the linking
numbers of the four generators with the knot.  The map is surjective because
$\gcd(p,q)=1$.
Lemma~\ref{lem:truncation-identities} shows that this homomorphism respects
the relator $r_1$.  It also respects $r_2$ and $r_3$, since
$\sum_{j=0}^{q-1}L_j=\sum_{k=0}^{p-1}M_k=r$.
Hence it factors through the knot group.  Since the first homology of a knot
complement is infinite cyclic, this surjective map is its abelianization, up
to the choice of orientation of $t$.

Retain the cumulative sums defined above, and recall that
$E_c=r^2s+pq$.  In the present range $p<r<p+q$, all indices in the
relevant suffix sums are positive, so $\widehat M_k=M_k$.
Lemma~\ref{lem:truncation-identities} therefore identifies the exponent
obtained from the inclusion word with $E_a$ defined above:
\[
rs\overline L_{j_*-1}+pj_*+r\lambda_*
=q(p-i_*+1)+rs\sum_{k=i_*}^{p-1}\widehat M_k
=E_a.
\]

To compute the Alexander polynomial, it is enough to use the $x$- and
$z$-columns of the Fox matrix together with the single $w$-entry coming
from $r_2$.  The required Fox derivatives are
\[
\frac{\partial r_1}{\partial x}
=
\sum_{m=0}^{j_*-1}
\left(\prod_{j=0}^{m-1}y^{L_js}x\right)y^{L_ms},
\]
\[
\frac{\partial r_1}{\partial z}
=
\varphi_1(a)\sum_{h=0}^{p-i_*}
\left(\prod_{\ell=1}^{h}zw^{M_{p-\ell}}\right),
\]
\[
\frac{\partial r_3}{\partial x}
=
\sum_{m=0}^{q-1}
\left(\prod_{j=0}^{m-1}y^{L_js}x\right)y^{L_ms},
\]
and
\[
\frac{\partial r_3}{\partial z}
=
\varphi_1(c)\sum_{m=0}^{p-1}
\left(\prod_{k=0}^{m-1}w^{M_k}z\right)w^{M_m}.
\]
Here and below, a product whose lower index exceeds its upper index is
understood to be empty.  Finally,
\[
\frac{\partial r_2}{\partial w}=y^s.
\]
After applying $\alpha$, we recover the four block polynomials from
Equation~\eqref{eq:block-polynomials}:
\[
X_a(t)=\alpha\!\left(\frac{\partial r_1}{\partial x}\right)
=\sum_{m=0}^{j_*-1}t^{rs\overline L_m+pm},
\]
\[
Z_a(t)=\alpha\!\left(\frac{\partial r_1}{\partial z}\right)
=t^{E_a}\sum_{h=0}^{p-i_*}
t^{-qh-rs\sum_{k=p-h}^{p-1}\widehat M_k},
\]
\[
X_c(t)=\alpha\!\left(\frac{\partial r_3}{\partial x}\right)
=\sum_{m=0}^{q-1}t^{rs\overline L_m+pm},
\]
and
\[
Z_c(t)=\alpha\!\left(\frac{\partial r_3}{\partial z}\right)
=t^{E_c}\sum_{m=0}^{p-1}t^{-rs\overline M_m-qm}.
\]
Also,
\[
\alpha\!\left(\frac{\partial r_2}{\partial w}\right)=t^{rs}.
\]

\begin{lemma}\label{lem:minor-formula}
Let $F(t)$ be a deficiency-one Alexander matrix of a knot group, and let
$m_v(t)$ be the maximal minor obtained by deleting the column associated
to a generator $v$.  Write $\alpha(v)=t^{e_v}$ and suppose that
$e_v\neq0$.  Then
\[
\Delta_K(t)\doteq
m_v(t)\frac{t-1}{t^{e_v}-1}.
\]
\end{lemma}

\begin{proof}
This is the standard deleted-column formula for a deficiency-one
presentation; see, for example,~\cite[Chapters~VIII--IX]{CrowellFox1977}.
In the present notation, it states that
\[
\frac{m_v(t)}{t^{e_v}-1}\doteq\frac{\Delta_K(t)}{t-1}.
\]
Rearranging gives the result.
\end{proof}

\begin{lemma}
\label{lem:old-new-compatibility}
Suppose that $1<r<p$ and $r\ne q$.  Let
$X^{\mathrm{old}},\widetilde X^{\mathrm{old}},
Y^{\mathrm{old}},\widetilde Y^{\mathrm{old}}$ be the four polynomials in
\cite[Theorem~1]{Park-Adnan-2025-01}.  Then
\[
X_a=t^{rs}\frac{\widetilde X^{\mathrm{old}}}{1-t^p},
\qquad
X_c=t^{rs}\frac{X^{\mathrm{old}}}{1-t^p},
\]
and
\[
Z_a=t^q\frac{\widetilde Y^{\mathrm{old}}}{1-t^q},
\qquad
Z_c=t^q\frac{Y^{\mathrm{old}}}{1-t^q}.
\]
\end{lemma}

\begin{proof}
Let $\boldsymbol{k}_i$ denote the cumulative integers denoted by $k_i$ in
\cite{Park-Adnan-2025-01}.  Comparing the residue intervals in that paper
with the $\mu_j$-blocks gives
\[
\boldsymbol{k}_i=0
\quad\Longleftrightarrow\quad
1\leq i\leq\overline L_0-1,
\]
and, for $1\leq m<q$,
\[
\boldsymbol{k}_i=m
\quad\Longleftrightarrow\quad
\overline L_{m-1}\leq i\leq\overline L_m-1.
\]
These intervals partition $\{1,\ldots,r-1\}$ because
$\overline L_{q-1}=r$.  Therefore
\[
t^{-rs}(1-t^p)X_c
=t^{rs(L_0-1)}+\sum_{m=1}^{q-1}
t^{pm+rs(\overline L_{m-1}-1)}(t^{rsL_m}-1)-t^{pq+(r-1)rs}.
\]
Expanding each factor $t^{rsL_m}-1$ into $L_m$ consecutive differences
shows that the right side is exactly $X^{\mathrm{old}}$.  Cutting the same
partition at the residue $Q'$ of~\cite{Park-Adnan-2025-01}, which is the
endpoint specified here by $j_*$, gives in the same way
\[
(1-t^p)X_a=t^{rs}\widetilde X^{\mathrm{old}}.
\]

For completeness, reading the residue partition modulo $p$ instead gives
the $\tau_i$-blocks.  Multiplication by $1-t^q$ cancels consecutive terms
inside these blocks; the full endpoint gives $Y^{\mathrm{old}}$, and the
endpoint $\tau_{i_*}=[r-q]_p$ gives $\widetilde Y^{\mathrm{old}}$.  Thus
\[
(1-t^q)Z_c=t^qY^{\mathrm{old}},
\qquad
(1-t^q)Z_a=t^q\widetilde Y^{\mathrm{old}}.
\]
The truncated endpoint exponents agree by the two identities
\[
\sum_{j=0}^{j_*}L_j=1+\sum_{k=i_*}^{p-1}M_k,
\qquad
pj_*+r=q(p-i_*+1).
\]
Both follow from the arithmetic argument in
Lemma~\ref{lem:truncation-identities}, which applies also when $r<p$.  This proves all four asserted identities.
\end{proof}

\subsection{Elementary cases outside the main formula}\label{sec:elementary-cases}

The restriction $q>1$ removes an elementary case: when $q=1$ and
$1<r\leq p+1$, the braid destabilizes to the torus knot $T(r,rs+1)$.
In particular, the strict complementary range is empty and the endpoint
$r=p+1$ is elementary.  Also, twisting a
single strand does not change the knot, so $T(p,q;1,s)=T(p,q)$.
For $r=1$, the computation uses the torus-knot Alexander polynomial.
For $q=1$, the residue convention $[a]_1=0$ gives $j_*=0$ in
Equation~\eqref{eq:unified-indices} and $\mu_0=0$.  For $r>1$,
the block sums therefore reduce to
\[
\Delta(t)\doteq
\frac{(1-t)(1-t^{r(rs+1)})}{(1-t^r)(1-t^{rs+1})},
\]
the Alexander polynomial of $T(r,rs+1)$.  All remaining cases in the
stated search range use Theorem~\ref{thm:alexander-polynomial}.

\subsection{Verification at the endpoint \texorpdfstring{$r=p+q$}{r=p+q}}

We show that the adjusted block lengths in the main theorem account for
the change in the terminal passage of $a$.

\begin{proposition}\label{prop:endpoint-formula}
Let $p>q>1$ be coprime and $r=p+q$.  Then $(j_*,i_*)=(q-1,0)$,
$\sum_{i=0}^{p-1}\widehat M_i=r-1$, and
$\widehat M_i=\widehat M_{p-1-i}$ for $0\leq i<p$.
Consequently, the block polynomials of Theorem~\ref{thm:alexander-polynomial}
satisfy
\[
E_a=rs(r-1)+q(p+1),\qquad
Z_a(t)=t^{E_a}\sum_{h=0}^{p}
 t^{-qh-rs\sum_{k=0}^{h-1}\widehat M_k},
\]
and the theorem holds at $r=p+q$ for every $s\in\mathbb Z$.
\end{proposition}

\begin{proof}
The index identities follow from Equation~\eqref{eq:unified-indices}.
At $r=p+q$, the residue definition gives $M_i=2$ exactly when
$\tau_i<q$, and $M_i=1$ otherwise.  In particular,
$M_0=2$ and $M_{p-1}=1$, so
$\widehat M_0=\widehat M_{p-1}=1$ and
$\sum\widehat M_i=r-1$.
For $1\leq i\leq p-2$, neither $\tau_i=0$ nor $\tau_i=q$ can occur, and
\[
\tau_{p-1-i}=[q-\tau_i]_p.
\]
Thus $\tau_{p-1-i}<q$ if and only if $\tau_i<q$, proving the palindrome
identity.  It follows that
\[
\sum_{k=p-h}^{p-1}\widehat M_k
=\sum_{k=0}^{h-1}\widehat M_k\qquad(0\leq h\leq p),
\]
which gives the displayed expression for $Z_a$.

The residue-block calculation of $\varphi_1(c)$ and $\varphi_2(c)$ is unchanged.
The truncated word on the first side is
\[
\varphi_1(a)=\left(\prod_{j=0}^{q-2}y^{L_js}x\right)
y^{(L_{q-1}-1)s+1}.
\]
On the second side, the final passage through the twisting disk is
replaced by a return across the disk dual to $z$.  Consequently,
\[
\varphi_2(a)=\left(\prod_{\ell=1}^{p-1}z^{-1}w^{-M_{p-\ell}}\right)
z^{-1}w^{-1}z^{-1},\qquad
\varphi_2(a)^{-1}=z\left(\prod_{k=0}^{p-1}w^{\widehat M_k}z\right).
\]
Products are read in increasing order of the displayed index.
The same Seifert--Van Kampen presentation holds with $r_1=\varphi_1(a)\varphi_2(a)^{-1}$,
$r_2=y^sw$, and $r_3=\varphi_1(c)\varphi_2(c)^{-1}$.
Since $\sum L_j=r$ and $p(q-1)+r=q(p+1)$,
\[
\alpha(\varphi_1(a))=\alpha(\varphi_2(a))=t^{rs(r-1)+q(p+1)}=t^{E_a},
\]
and the abelianization remains primitive.  Differentiating gives
\[
\frac{\partial r_1}{\partial z}
=\varphi_1(a)\left(1+\sum_{h=1}^{p}
 z\left(\prod_{k=0}^{h-2}w^{\widehat M_k}z\right)
 w^{\widehat M_{h-1}}\right).
\]
Its abelianization is the displayed $Z_a$.  The other three derivatives
remain $X_a,X_c,Z_c$.  Thus the minor obtained by deleting the $y$-column
is $t^{rs}(Z_aX_c-X_aZ_c)$, and Lemma~\ref{lem:minor-formula}
completes the proof of the main theorem at the endpoint.
\end{proof}

\subsection{Proof of the main theorem}

\begin{proof}[Proof of Theorem~\ref{thm:alexander-polynomial}]
Assume first that $p<r<p+q$.  Let $F(t)$ be the Alexander matrix associated
to the presentation in Theorem~\ref{thm:knot-group}.  Delete the column
corresponding to $y$ and denote the resulting determinant by $m_y(t)$.
Expanding along the row corresponding to $r_2=y^sw$ gives
\[
m_y(t)
=t^{rs}\bigl(Z_a(t)X_c(t)-X_a(t)Z_c(t)\bigr).
\]
Since $\alpha(y)=t^r$, Lemma~\ref{lem:minor-formula} gives
\[
\Delta_{T(p,q;r,s)}(t)
\doteq
m_y(t)\frac{1-t}{1-t^r}.
\]
The factor $t^{rs}$ is a unit in $\mathbb Z[t^{\pm1}]$, so substituting the
formula for $m_y(t)$ proves the result for every $s\in\mathbb Z$.
Negative values of $s$ only produce negative exponents, which are allowed
in the Laurent polynomial ring.

Now suppose that $1<r<p$ and $r\ne q$.  By
Lemma~\ref{lem:old-new-compatibility},
\[
Z_aX_c-X_aZ_c
=-\frac{t^{q+rs}}{(1-t^p)(1-t^q)}
\bigl(\widetilde X^{\mathrm{old}}Y^{\mathrm{old}}
-X^{\mathrm{old}}\widetilde Y^{\mathrm{old}}\bigr).
\]
The monomial factor is a unit, so the formula in
\cite[Theorem~1]{Park-Adnan-2025-01} becomes the formula stated here.

It remains to treat $r=q$ and $r=p$.  When $r=q$,
Equation~\eqref{eq:unified-indices} gives
$(j_*,i_*)=(0,p)$, and direct substitution into
Equation~\eqref{eq:block-polynomials} yields
\[
\frac{1-t}{1-t^q}(Z_aX_c-X_aZ_c)
\doteq
\frac{(1-t)(1-t^{q(p+qs)})}
{(1-t^q)(1-t^{p+qs})},
\]
the Alexander polynomial of $T(q,p+qs)$.  When $r=p$, one has
$(j_*,i_*)=(q-1,1)$, and the same substitution gives
\[
\frac{1-t}{1-t^p}(Z_aX_c-X_aZ_c)
\doteq
\frac{(1-t)(1-t^{p(q+ps)})}
{(1-t^p)(1-t^{q+ps})},
\]
the Alexander polynomial of $T(p,q+ps)$.  These are precisely
$T(p,q;q,s)$ and $T(p,q;p,s)$, respectively.  Finally, the endpoint
$r=p+q$ is covered by Proposition~\ref{prop:endpoint-formula}.
\end{proof}

\section{Monicity, Alexander breadth, and Seifert genus}

This section analyzes the extreme terms of the block polynomials in the
complementary range.  We first treat the endpoint $r=p+q$, and then prove
monicity and determine the Alexander breadth in the strict complementary
range.  The resulting genus bound is sharp for $s\geq0$ by fiberedness.
Monicity is proved before the breadth formula because its extreme-term
analysis is used in that calculation.

\subsection{Applications at the endpoint \texorpdfstring{$r=p+q$}{r=p+q}}

\begin{proposition}\label{prop:endpoint-applications}
Let $p>q>0$ be coprime and let $r=p+q$.  Then
$\Delta_{T(p,q;r,s)}(t)$ is monic and
\[
\operatorname{br}\Delta_{T(p,q;r,s)}(t)=
\begin{cases}
(p-1)(q-1)+sr(r-1),&s\geq0,\\[2mm]
(-s)r(r-1)-(p-1)(q-1)-2(r-1),&s<0.
\end{cases}
\]
If $s\geq0$, then $T(p,q;r,s)$ is fibered and its genus is one half of
the corresponding breadth.
\end{proposition}

\begin{proof}
If $q=1$, the elementary-case formula identifies the knot with
$T(r,rs+1)$, and the assertions follow from the torus-knot formulas.
Assume that $q>1$.  Put
\[
\xi=sr^2+p(q-1).
\]
Since $j_*=q-1$, we have $X_c=X_a+t^\xi$.  Comparing the expression for
$Z_a$ in Proposition~\ref{prop:endpoint-formula} with $Z_c$ gives
$Z_a=Z_c+t^{E_a}$.  Hence
\begin{equation}\label{eq:endpoint-reduced-determinant}
Z_aX_c-X_aZ_c=t^{E_a}X_c+t^\xi Z_c.
\end{equation}

Suppose first that $s>0$.  The exponents in $X_c$ increase with the
summation index, whereas those in $Z_c$ decrease.  The difference between
the lowest possible exponent of $t^{E_a}X_c$ and that of $t^\xi Z_c$ is
\[
p+rs(L_0-1)>0,
\]
and the difference between the highest possible exponent of
$t^{E_a}X_c$ and that of $t^\xi Z_c$ is
\[
q+rs(M_0-1)>0.
\]
Thus the unique lowest term of the right side of
Equation~\eqref{eq:endpoint-reduced-determinant} has exponent $\xi+q$,
and its unique highest term has exponent $E_a+\xi$.  Both coefficients
are $1$, and therefore
\[
\operatorname{br}(Z_aX_c-X_aZ_c)=E_a-q=pq+sr(r-1).
\]

Now let $s=-u<0$.  The monotonicities reverse.  The differences that
separate the two possible lowest and highest terms are respectively
\[
ur(M_0-1)-q>0,
\qquad
ur(L_0-1)-p>0.
\]
Here $M_0=2$ and $L_0-1\geq2$ at the endpoint.  Consequently the unique
lowest term has exponent $E_a+\xi$, the unique highest term has exponent
$\xi+q$, and
\[
\operatorname{br}(Z_aX_c-X_aZ_c)=q-E_a=ur(r-1)-pq.
\]
In both cases, subtracting $r-1$ for the factor
$1+t+\cdots+t^{r-1}$ gives the asserted breadth.  The unique extreme
terms also show that the quotient is monic.  The case $s=0$ is the
torus-knot specialization.

Finally, for $s>0$, use the standard positive $(p+q)$-strand Lorenz braid
for $T(p,q)$ and append $s$ positive full twists on all $r=p+q$ strands.
This is a positive braid with
\[
c=pq+sr(r-1)
\]
crossings.  Its closure is fibered, and its canonical fiber surface has
\[
2g=c-r+1=(p-1)(q-1)+sr(r-1).
\]
For $s=0$, the same conclusion is the usual torus-knot genus formula.
\end{proof}

\begin{remark}
Proposition~\ref{prop:endpoint-applications} extends the breadth and
monicity theorems to $r=p+q$.  The searches reported in the Introduction
still exclude this endpoint, so their stated counts are unchanged.
\end{remark}

\subsection{Monicity in the strict complementary range}

The endpoint $r=p+q$ was established in
Proposition~\ref{prop:endpoint-applications}.  We therefore assume
$p<r<p+q$ throughout this subsection and define the complementary sums
\[
X_+(t)=\sum_{m=j_*}^{q-1}t^{rs\overline L_m+pm}
\]
and
\[
Z_-(t)=t^{E_c}\sum_{m=0}^{i_*-2}
t^{-rs\overline M_m-qm}.
\]
The bounds $j_*\leq q-2$ and $i_*\geq2$ proved in
Section~2 ensure that both displayed sums are
well-defined and that $X_+(t)$ is nonzero.

\begin{proof}[Proof of Theorem~\ref{thm:monicity}]
We first separate the two determinant terms in
Theorem~\ref{thm:alexander-polynomial}.  By definition,
\[
X_c(t)=X_a(t)+X_+(t).
\]
We next claim that
\[
Z_c(t)=Z_-(t)+Z_a(t).
\]
Indeed, the term of $Z_c(t)$ indexed by $m=i_*-1$ has exponent
\[
E_c-rs\overline M_{i_*-1}-q(i_*-1)
=q(p-i_*+1)+rs\sum_{k=i_*}^{p-1}M_k
=E_a.
\]
For $0\leq h<p-i_*$, the change in exponent between the terms of $Z_a(t)$
indexed by $h$ and $h+1$ is
\[
-q-rsM_{p-1-h},
\]
whereas the corresponding change between the terms of $Z_c(t)$ indexed by
$i_*-1+h$ and $i_*+h$ is
\[
-q-rsM_{i_*+h}.
\]
These changes agree by Lemma~\ref{lem:M-palindrome}.  This proves the claim.
Consequently,
\begin{equation}\label{eq:reduced-determinant}
Z_aX_c-X_aZ_c=Z_aX_+-X_aZ_-.
\end{equation}

Suppose first that $s=-u<0$.  Consecutive exponents in each of $X_a$ and
$X_+$ differ by
\[
p-ruL_{m+1}<0,
\]
and consecutive exponents in $Z_c$ differ by
\[
ruM_{m+1}-q>0.
\]
Likewise, as the summation index in $Z_a$ increases, its exponents increase
by $ruM_{p-1-h}-q>0$.  Thus every extreme term used below is unique.  The
highest term of $Z_a$ has exponent $q$, and the highest term of $X_+$ is its
term indexed by $j_*$.  If $j_*=0$, then $X_a=0$, so
Equation~\eqref{eq:reduced-determinant} already has a unique highest term of
coefficient $1$.

Assume that $j_*>0$.  Put
\[
A=\sum_{j=0}^{j_*}L_j,
\qquad
B=\sum_{k=i_*}^{p-1}M_k.
\]
Lemma~\ref{lem:truncation-identities} gives $A=B+1$ and
$pj_*+r=q(p-i_*+1)$.  The term of $Z_c$ indexed by $i_*-1$ has exponent
$E_a$, so the highest term of $Z_-$, indexed by $i_*-2$, has exponent
\[
E_a+q-ruM_{i_*-1}.
\]
It follows that
\[
\deg(Z_aX_+)-\deg(X_aZ_-)
=r\bigl(u(L_0+M_{i_*-1}-1)-1\bigr)>0.
\]
Here $L_0=\lceil r/q\rceil\geq2$ and $M_{i_*-1}\geq1$.  Hence the highest
term on the right side of Equation~\eqref{eq:reduced-determinant} is unique
and has coefficient $1$.

Now suppose that $s>0$.  The same monotonicity argument, with all directions
reversed, applies to the lowest terms.  If $j_*=0$, the conclusion is again
immediate.  If $j_*>0$, then
\[
\operatorname{mindeg}(X_aZ_-)-\operatorname{mindeg}(Z_aX_+)
=r\bigl(1+s(L_0+M_{i_*-1}-1)\bigr)>0.
\]
Thus the lowest term in Equation~\eqref{eq:reduced-determinant} is unique and
has coefficient $1$.

Finally,
\[
\frac{1-t^r}{1-t}=1+t+\cdots+t^{r-1}
\]
is monic and has constant term $1$.  Theorem~\ref{thm:alexander-polynomial}
therefore shows that the Alexander polynomial has an extreme coefficient of
absolute value $1$.  Since the Alexander polynomial of a knot is symmetric,
it is monic.  The remaining case $s=0$ follows from the torus-knot
specialization below.
\end{proof}

\begin{remark}[The case $s=0$]
When $s=0$, the twisted torus knot is the torus knot $T(p,q)$.  The formula
in Theorem~\ref{thm:alexander-polynomial} recovers its Alexander polynomial
directly.  Indeed, put $\ell_*=p-i_*+1$.  Lemma~\ref{lem:truncation-identities}
gives $q\ell_*=pj_*+r$, and the four sums become
\[
X_a(t)=\frac{1-t^{pj_*}}{1-t^p},
\qquad
X_c(t)=\frac{1-t^{pq}}{1-t^p},
\]
\[
Z_a(t)=t^q\frac{1-t^{q\ell_*}}{1-t^q},
\qquad
Z_c(t)=t^q\frac{1-t^{pq}}{1-t^q}.
\]
It follows that
\[
Z_a(t)X_c(t)-X_a(t)Z_c(t)
=
t^{q+pj_*}
\frac{(1-t^{pq})(1-t^r)}{(1-t^p)(1-t^q)}.
\]
Therefore,
\[
\Delta_{T(p,q;r,0)}(t)
\doteq
\frac{(1-t)(1-t^{pq})}{(1-t^p)(1-t^q)},
\]
which is the Alexander polynomial of $T(p,q)$.
\end{remark}

\subsection{Alexander degree and Seifert genus}

We next determine the breadth of the Alexander polynomial, or equivalently
the degree after normalizing its lowest exponent to zero.  This gives the
Seifert genus for $s\geq0$ and an explicit lower bound for $s<0$.

\begin{proof}[Proof of Theorem~\ref{thm:alexander-degree}]
The endpoint $r=p+q$ follows from
Proposition~\ref{prop:endpoint-applications}, so it remains to assume
$p<r<p+q$.
Put
\[
P(t)=Z_a(t)X_+(t),
\qquad
R(t)=X_a(t)Z_-(t),
\]
so that the numerator in Equation~\eqref{eq:reduced-determinant} is
$N(t)=P(t)-R(t)$.  Set $A=\sum_{j=0}^{j_*}L_j$,
$B=\sum_{k=i_*}^{p-1}M_k$, and $H=p-i_*$.  If $j_*=0$, then $R=0$
and $\operatorname{br}N=\operatorname{br}P$.  Suppose now that $j_*>0$.
The first and last exponents in the defining sums are as follows:
\[
\begin{array}{c|c|c}
 &\text{first exponent}&\text{last exponent}\\ \hline
Z_a&E_a&q\\
X_+&rsA+pj_*&rsr+p(q-1)\\
X_a&rsL_0&rs(A-L_{j_*})+p(j_*-1)\\
Z_-&E_c-rsM_0&E_a+q+rsM_{i_*-1}
\end{array}
\]
For $s>0$, the exponents increase in the $X$-sums and decrease in the
$Z$-sums; for $s<0$, these directions reverse.  Thus, for example,
\[
\begin{aligned}
\operatorname{maxdeg}P-\operatorname{maxdeg}R
={}&E_a+rsr+p(q-1)\\
 &{}-rs(A-L_{j_*})-p(j_*-1)-E_c+rsM_0\\
={}&r\bigl(1+s(L_{j_*}+M_0-1)\bigr)
\end{aligned}
\]
when $s>0$.  Here we used $A=B+1$ and
$pj_*+r=q(p-i_*+1)$ from Lemma~\ref{lem:truncation-identities}.
When $s=-u<0$, the opposite endpoint comparison is
\[
\begin{aligned}
\operatorname{mindeg}R-\operatorname{mindeg}P
={}&-ru(A-L_{j_*})+p(j_*-1)+E_c+ruM_0\\
 &{}-E_a+rur-p(q-1)\\
={}&r\bigl(u(L_{j_*}+M_0-1)-1\bigr).
\end{aligned}
\]
Together with the two comparisons already proved for monicity, these give,
for $s>0$,
\begin{align*}
\operatorname{mindeg}R-\operatorname{mindeg}P
&=r\bigl(1+s(L_0+M_{i_*-1}-1)\bigr)>0,\\
\operatorname{maxdeg}P-\operatorname{maxdeg}R
&=r\bigl(1+s(L_{j_*}+M_0-1)\bigr)>0.
\end{align*}
For $s=-u<0$,
\begin{align*}
\operatorname{maxdeg}P-\operatorname{maxdeg}R
&=r\bigl(u(L_0+M_{i_*-1}-1)-1\bigr)>0,\\
\operatorname{mindeg}R-\operatorname{mindeg}P
&=r\bigl(u(L_{j_*}+M_0-1)-1\bigr)>0.
\end{align*}
The last quantity is positive because $M_0=\lceil r/p\rceil=2$.
Consequently, when $j_*>0$, the exponent interval of $R$ lies strictly
inside that of $P$.  Hence, in all cases,
\[
\operatorname{br}N=\operatorname{br}P
=\operatorname{br}Z_a+\operatorname{br}X_+.
\]
Lemma~\ref{lem:truncation-identities} gives $A=B+1$ and
$pj_*+r=q(H+1)$.  If $s>0$, monotonicity of the exponents gives
\begin{align*}
\operatorname{br}Z_a&=qH+rsB,\\
\operatorname{br}X_+&=p(q-1-j_*)+rs(r-A).
\end{align*}
Therefore
\[
\operatorname{br}N=p(q-1-j_*)+qH+rs(r-1)
=(p-1)(q-1)+(r-1)+s r(r-1).
\]
For $s=-u<0$, the reversed monotonicity gives
\[
\operatorname{br}N=(ruB-qH)+\bigl(ru(r-A)-p(q-1-j_*)\bigr)
=ur(r-1)-(p-1)(q-1)-(r-1).
\]
Since
\[
N(t)=(1+t+\cdots+t^{r-1})\Delta_K(t)
\]
up to a unit, subtracting $r-1$ from the breadth of $N$ proves both
formulas.  The case $s=0$ follows from the torus-knot specialization.

For the fiberedness assertion, we use the $r$-strand braid presentation
in Lee~\cite[proof of Lemma~2.4, Figures~3(b) and~4]{Lee-2015}, which
explicitly treats $r>p$.  The isotopy in Figure~4 carries the twisting
circle to a circle around all $r$ braid strands.  In Figure~3(b), the
resulting braid consists of the full-twist block $(r,rs)$, a torus-braid
block $(p+q-r,p+q-r)$, and two positive crossings of parallel packets.
The latter contribute $(r-p)p$ and $(p+q-r)(r-q)$ crossings,
respectively.  All these blocks are positive for $s>0$.  Although Lee
uses this diagram with $s=2,3$ in that proof, the isotopy is independent
of $s$, and positivity requires only $s>0$.  Thus $K$ is the closure of
a positive $r$-strand braid, hence is fibered by
Stallings' theorem~\cite{Stallings-1978}.  The total crossing number of
this braid is $c=sr(r-1)+pq-p-q+r$, so its canonical fiber surface gives
\[
2g(K)=c-r+1=(p-1)(q-1)+sr(r-1).
\]
This agrees with the breadth formula above and independently verifies the
genus formula for $s>0$.  For $s=0$, the same genus formula holds for the
fibered torus knot $K=T(p,q)$.
\end{proof}

\section{Non-\texorpdfstring{$L$}{L}-space families in the complementary range}

We now use the Alexander-polynomial formula and the reduced determinant
obtained above to study interior coefficients.  Unlike the arguments of the
preceding section, which concern only the extreme terms, the $L$-space
obstruction requires a more detailed coefficient calculation.  For an
$L$-space knot, all nonzero Alexander-polynomial coefficients are $\pm1$ and
alternate in sign; hence a coefficient of absolute value greater than one
obstructs the $L$-space property.  We use the reduced determinant in
Equation~\eqref{eq:reduced-determinant} to locate the first such coefficient
for both signs of the twisting parameter.  To obtain uniform two-parameter
families, we specialize to $(p,q,r)=(2n+1,2n,3n)$.  In this case
$i_*=j_*=n$, and the four factors in the reduced determinant become
geometric sums, making the coefficient calculation explicit.  The position
and sign of the first obstructing coefficient will also distinguish the
knots when the Alexander breadth alone does not.  The same coefficient-
extraction procedure applies throughout $p<r<p+q$; simple block patterns
such as these make uniform coefficient obstructions accessible.

\begin{lemma}\label{lem:family-initial-coefficients}
For $K=T(2n+1,2n;3n,s)$, normalize its Alexander polynomial
$Q(t)=\sum_{j\geq0}a_jt^j$ to have lowest exponent zero and positive
constant term.  Then the first coefficient of absolute value greater
than one is as follows:
\[
\begin{array}{c|c|c}
\text{parameters}&J=\min\{j:|a_j|>1\}&a_J\\ \hline
s=-u,\quad u\geq1,\ n\geq3&4(3u-2)n&2\\
s=u,\quad u\geq2,\ n\geq4&(15u+9)n+3&-2\\
s=1,\quad n\geq5&29n+4&-2
\end{array}
\]
\end{lemma}

\begin{proof}
The indices and block lengths in Theorem~\ref{thm:alexander-polynomial} are
\[
\begin{aligned}
j_*&=i_*=n,\\
(L_0,\ldots,L_{2n-1})
&=(2,\underbrace{1,\ldots,1}_{n},\underbrace{2,\ldots,2}_{n-1}),\\
(M_0,\ldots,M_{2n})
&=(\underbrace{2,\ldots,2}_{n-1},\underbrace{1,\ldots,1}_{n+2}).
\end{aligned}
\]
For an integer $m\geq0$, write
$\mathcal S_m(v)=\sum_{h=0}^m t^{hv}$.  Reading the four sums in
Equation~\eqref{eq:reduced-determinant} with these block lengths gives
\[
\begin{aligned}
Z_a&=t^q\mathcal S_{n+1}(rs+q),\\
X_+&=t^{rs(n+2)+pn}\mathcal S_{n-1}(2rs+p),\\
X_a&=t^{2rs}\mathcal S_{n-1}(rs+p),\\
Z_-&=t^{E_c-2rs}\mathcal S_{n-2}(-2rs-q).
\end{aligned}
\]
Let $N=Z_aX_+-X_aZ_-$ and define the shifts
\[
E_-=(6-12u)n^2+(4-3u)n-1,\qquad
E_+=(3u+2)n^2+(6u+3)n.
\]
For $s=-u$, reverse the sums with negative steps and set
\[
d_-=(3u-2)n,\quad e_-=(6u-2)n-1,\quad \eta_-=(6u-3)n.
\]
This gives the exact identity
\[
F_-:=t^{-E_-}N
=\mathcal S_{n+1}(d_-)\mathcal S_{n-1}(e_-)
-t^{\eta_-}\mathcal S_{n-1}(d_--1)\mathcal S_{n-2}(e_-+1).
\]
For $s=u>0$, set
\[
d_+=(3u+2)n,\quad e_+=(6u+2)n+1,\quad \eta_+=(6u+3)n.
\]
Reversing the sum in $Z_-$ instead gives
\[
F_+:=t^{-E_+}N
=\mathcal S_{n+1}(d_+)\mathcal S_{n-1}(e_+)
-t^{\eta_+}\mathcal S_{n-1}(d_++1)\mathcal S_{n-2}(e_+-1).
\]
All steps and shifts in these two expressions are positive, and the
constant term is $1$.  Thus $E_\pm$ are indeed the lowest exponents,
and $F_\pm=(1+t+\cdots+t^{3n-1})Q$.

Here are all numerator terms needed through the corresponding degree
$J$ in the statement; each congruence is modulo $t^{J+1}$:
\begin{align}
F_-&\equiv\mathcal S_4(d_-)+t^{e_-}\mathcal S_1(d_-)
-t^{\eta_-}\mathcal S_1(d_--1),
&&J=4d_-;\label{eq:negative-numerator-initial}\\
F_+&\equiv\mathcal S_4(d_+)+t^{e_+}\mathcal S_3(d_+)
+t^{2e_+}\mathcal S_1(d_+)\notag\\
&\quad-t^{\eta_+}\bigl(\mathcal S_3(d_++1)
+t^{e_+-1}\mathcal S_1(d_++1)\bigr),
&&J=(15u+9)n+3,\ u\geq2;\label{eq:positive-numerator-initial}\\
F_+&\equiv\mathcal S_5(d_+)+t^{e_+}\mathcal S_4(d_+)
+t^{2e_+}\mathcal S_2(d_+)+t^{3e_+}\mathcal S_1(d_+)\notag\\
&\quad-t^{\eta_+}\bigl(\mathcal S_4(d_++1)
+t^{e_+-1}\mathcal S_2(d_++1)+t^{2(e_+-1)}\bigr),
&&J=29n+4,\ u=1.\label{eq:positive-one-numerator-initial}
\end{align}
These lists follow by fixing the index of the second geometric sum and
bounding the first index.  In the negative case the allowed first indices
are $0,\ldots,4$ for the first product at second index $0$, at most
$0,1$ at second index $1$, and at most $0,1$ for the second product at
second index $0$.  All other terms exceed $J$; when $u=1$, the term
$t^{e_-+d_-}$ also exceeds $J$.
For $u\geq2$ in the positive case, the first product has first-index
bounds $4,3,1$ at second indices $0,1,2$, and the second product has bounds
$3,1$ at second indices $0,1$.
For $u=1$, these bounds are respectively $5,4,2,1$ at second indices
$0,1,2,3$, and $4,2,0$ at second indices $0,1,2$.
The hypotheses on $n$ ensure that all these indices belong to the original
finite sums.
For completeness, lower bounds for the exponents of terms omitted from
these expressions, in the first and second products respectively, are
\[
\begin{array}{c|c|c}
 &\text{first product}&\text{second product}\\ \hline
s=-u&\min\{5d_-,e_-+2d_-\}&\eta_-+2(d_--1)\\
s=u,\ u\geq2&5d_+&\eta_++2(e_+-1)\\
s=1&6d_+&\eta_++2(e_+-1)+d_++1
\end{array}
\]
Subtracting the corresponding $J$ gives, row by row,
\[
\bigl(\min\{d_-,2n-1\},\,n-2\bigr),\qquad
\bigl(n-3,\,(3u-2)n-3\bigr),\qquad
\bigl(n-4,\,n-3\bigr),
\]
whose entries are positive under the stated hypotheses.  These bounds
remain valid if a candidate omitted index lies beyond the original finite
sum, since excluding such a term can only raise the minimum.

To check both the first obstruction and the absence of earlier ones,
write $F_\pm=\sum b_jt^j$, with all negative-index coefficients zero.
Then
\begin{equation}\label{eq:non-l-space-recurrence}
a_j=b_j-b_{j-1}+a_{j-3n}
=\sum_{k\geq0}(b_{j-3nk}-b_{j-3nk-1}).
\end{equation}
The following tables make this coefficient calculation explicit.  For a
residue $0\leq\rho<3n$, the entry $\kappa^+$ denotes an increment $+1$
in $a_{3n\ell+\rho}$ at $\ell=\kappa$, and $\kappa^-$ denotes an
increment $-1$; a superscript $-2$ denotes an increment $-2$.
Increments at the same time are added, and those whose degree exceeds
$J$ are ignored.  Each coefficient is the cumulative sum of these
increments, starting from zero.  Every residue not listed has zero
coefficients through $J$.

For $s=-u$, expanding Equation~\eqref{eq:negative-numerator-initial}
and its multiple by $-t$ gives
\[
\begin{array}{c|l}
\rho&\text{increment times}\\ \hline
0&(0)^{+},\ (2u-1)^{-},\ (3u-2)^{+}\\
1&(0)^{-},\ (2u-1)^{+},\ (3u-2)^{-}\\
n-1&(2u-1)^{+},\ (3u-2)^{-}\\
n&(u-1)^{+},\ (2u-1)^{-},\ (3u-2)^{+},\ (4u-3)^{+}\\
n+1&(u-1)^{-},\ (4u-3)^{-}\\
2n-1&(3u-2)^{+}\\
2n&(2u-2)^{+},\ (3u-2)^{-}\\
2n+1&(2u-2)^{-}\\
\end{array}
\]
Since $n\geq3$, the listed residues are distinct and lie in $[0,3n)$.
Before $J$, the cumulative sums in each row lie in $\{0,\pm1\}$.
At $J=3n(4u-3)+n$, the row $\rho=n$ first reaches $2$.
For $u=1$, the increments at coincident times cancel as indicated by the
table; equivalently, $t^{3d_-}$ cancels $t^{\eta_-}$ and $t^{e_-}$ cancels
$t^{\eta_-+d_--1}$ in the numerator.

For $s=u\geq2$, Equation~\eqref{eq:positive-numerator-initial} gives
\[
\begin{array}{c|l}
\rho&\text{increment times}\\ \hline
0&(0)^{+},\ (2u+1)^{-},\ (3u+2)^{+}\\
1&(0)^{-},\ (2u+1)^{+},\ (3u+2)^{-},\ (4u+2)^{+}\\
2&(4u+2)^{-},\ (5u+2)^{+}\\
3&(5u+2)^{-},\ (5u+3)^{-}\\
n&(2u+1)^{+}\\
n+1&(2u+1)^{-},\ (3u+1)^{+},\ (5u+2)^{-}\\
n+2&(3u+1)^{-},\ (4u+1)^{+},\ (4u+2)^{-},\ (5u+2)^{+}\\
n+3&(4u+1)^{-},\ (4u+2)^{+}\\
2n&(u)^{+},\ (4u+1)^{-},\ (4u+2)^{+}\\
2n+1&(u)^{-},\ (2u)^{+},\ (3u+1)^{-},\ (4u+1)^{+},\ (4u+2)^{-},\ (5u+2)^{+}\\
2n+2&(2u)^{-},\ (3u+1)^{+},\ (5u+2)^{-}\\
\end{array}
\]
Since $n\geq4$, the listed residues are distinct and lie in $[0,3n)$.
The entries in each row are in increasing order for $u\geq2$.
Every row has cumulative sums in $\{0,\pm1\}$ until
$J=3n(5u+3)+3$, where the row $\rho=3$ reaches $-2$.
The omitted increment at $\rho=4$, $\kappa=5u+3$, is at degree $J+1$
and does not affect the calculation, including when $n=4$.

Finally, for $s=1$, Equation~\eqref{eq:positive-one-numerator-initial}
gives
\[
\begin{array}{c|l|c|l}
\rho&\text{increment times}&\rho&\text{increment times}\\ \hline
0&(0)^{+},\ (3)^{-},\ (5)^{+}&n+3&(5)^{-},\ (6)^{+}\\
1&(0)^{-},\ (3)^{+},\ (5)^{-},\ (6)^{+}&2n&(1)^{+},\ (5)^{-},\ (6)^{+}\\
2&(6)^{-},\ (7)^{+},\ (9)^{-}&2n+1&(1)^{-},\ (2)^{+},\ (4)^{-},\ (5)^{+},\ (6)^{-},\ (7)^{+}\\
3&(7)^{-},\ (9)^{+}&2n+2&(2)^{-},\ (4)^{+},\ (7)^{-},\ (8)^{+}\\
n&(3)^{+}&2n+3&(8)^{-},\ (9)^{+}\\
n+1&(3)^{-},\ (4)^{+},\ (7)^{-},\ (9)^{+}&2n+4&(9)^{-2}\\
n+2&(4)^{-},\ (5)^{+},\ (6)^{-},\ (7)^{+},\ (9)^{-}&&\\
\end{array}
\]
Again all cumulative sums before $J$ lie in $\{0,\pm1\}$, while the
row $\rho=2n+4$ reaches $-2$ at $\kappa=9$, or $J=29n+4$.
The next increment is at $J+1$ and is irrelevant.  Since $n\geq5$,
all listed residues are distinct and lie in $[0,3n)$.
This proves all three assertions.
\end{proof}

\begin{proof}[Proof of Theorem~\ref{thm:non-l-space-family}]
Lemma~\ref{lem:family-initial-coefficients} and the coefficient obstruction
prove the non-$L$-space assertions.  The opposite signs of the first
obstructing coefficients distinguish the two twist signs.  To distinguish
parameters within each sign, define
\[
\nu(\Delta)=\min\{j:|[t^j]\Delta^{\mathrm{norm}}(t)|>1\},
\]
where $\Delta^{\mathrm{norm}}$ has lowest exponent zero and positive
constant term, and $[t^j]f(t)$ denotes the coefficient of $t^j$ in $f(t)$.
For the negative-twist family, set $x=\nu/4=n(3u-2)$.  Theorem~\ref{thm:alexander-degree} becomes
\[
\operatorname{br}\Delta=(3n-1)x+2n^2-6n+2.
\]
For fixed $x$, the right side at $n+1$ minus its value at $n$ is
$3x+4n-4>0$.  Hence the polynomial determines $n$, and then
$u=(x/n+2)/3$.

For the positive-twist family with $u\geq2$, set
$x=(\nu-3)/3=n(5u+3)$.  The breadth formula becomes
\[
\operatorname{br}\Delta
=\frac{(9n-3)x-7n^2-n}{5}.
\]
For every admissible $n$ we have $x\geq13n$, so its numerator increases
by $9x-14n-8>0$ when $n$ is replaced by $n+1$,
so again $n$, and then $u=(x/n-3)/5$, is determined.  In the subfamily $u=1$, write its parameter as $n'$.  Then
$\nu=29n'+4$ determines $n'$ directly.  A polynomial
from this subfamily cannot equal one from the $u\geq2$ subfamily: equating
both $\nu$ and the breadth first gives
$29n'+4=(15u+9)n+3$.  Eliminating $n'$ from this identity and
the breadth equality gives
\[
n^2(2925u^2-4059u-2311)+n(143-42u)+158=0,
\]
whose left side, upon setting $v=u-2\geq0$, is
\[
2925n^2v^2+n(7641n-42)v+1271n^2+59n+158>0.\qedhere
\]
\end{proof}

\bibliographystyle{alpha}
\bibliography{references-ttk}

\end{document}